\documentclass{jams-l}
\newcommand{\R}{\mathbb{R}}

\usepackage{enumerate}
\usepackage{graphicx}
\usepackage{xcolor,tikz}

\usetikzlibrary{decorations,decorations.markings} 
\pgfkeys{/tikz/.cd,
    contour distance/.store in=\ContourDistance,
    contour distance=-10pt, % for the other orientation use a +
    contour step/.store in=\ContourStep,
    contour step=1pt,
}

\pgfdeclaredecoration{closed contour}{initial}
{% 
\state{initial}[width=\ContourStep,next state=cont] {
    \pgfmoveto{\pgfpoint{\ContourStep}{\ContourDistance}}
    \pgfcoordinate{first}{\pgfpoint{\ContourStep}{\ContourDistance}}
    \pgfpathlineto{\pgfpoint{0.3\pgflinewidth}{\ContourDistance}}
    \pgfcoordinate{lastup}{\pgfpoint{1pt}{\ContourDistance}}
    
  }
  \state{cont}[width=\ContourStep]{
     \pgfmoveto{\pgfpointanchor{lastup}{center}}
     \pgfpathlineto{\pgfpoint{\ContourStep}{\ContourDistance}}
     \pgfcoordinate{lastup}{\pgfpoint{\ContourStep}{\ContourDistance}}
  }
  \state{final}[width=\ContourStep]
  { % perhaps unnecessary but doesn't hurt either
    \pgfmoveto{\pgfpointanchor{lastup}{center}}
    \pgfpathlineto{\pgfpointanchor{first}{center}}
  }
}

\usepackage{hyperref}       % hyperlinks

\newtheorem{theorem}{Theorem}
\newtheorem*{theorem*}{Theorem}
\newtheorem{lemma}[theorem]{Lemma}

\newtheorem{claim}[theorem]{Claim}
\newtheorem*{obs*}{Observation}
\newtheorem*{cor*}{Corollary}
\newtheorem{prop}[theorem]{Proposition}

\newtheorem*{corollary*}{Corollary}
\newtheorem*{remark*}{Remark}
\newtheorem*{claim*}{Claim}
\newtheorem*{lemma*}{Lemma}
\newtheorem*{prop*}{Proposition}

\newtheorem*{remark}{Remark}

\newcommand{\ip}[2]{\langle #1 , #2 \rangle}

\newcommand{\cU}{\mathcal{U}}

\newcommand{\conv}{\mathsf{conv}}

\newcommand{\eps}{\varepsilon}

\newcommand{\GM}{\mathsf{GM}}

\numberwithin{equation}{section}

\begin{document}

\title{Intuitive norms are Euclidean}


%    author one information

\author{Shay Moran}
\address{Departments of Mathematics, Computer Science, Data and Decision Sciences, Technion-IIT, and Google Research.}
\email{smoran@technion.ac.il}

\author{Alexander Shlimovich}
\address{Department of Mathematics, Technion-IIT}
\email{ashlimovich@campus.technion.ac.il}

\author{Amir Yehudayoff}
\address{Department of Computer Science, The University of Copenhagen,
and Department of Mathematics, Technion-IIT}
\email{amir.yehudayoff@gmail.com}

\thanks{A.Y. is supported by a DNRF chair grant,
and partially supported by BARC}

\begin{abstract}
We call a norm on $\R^n$ intuitive 
if for every points $p_1,\ldots,p_m$ in~$\R^n$,
one of the geometric medians of the points over the norm
is in their convex hull. 
We characterize all intuitive norms. 
\end{abstract}

\maketitle

\section{Introduction}

The starting point of our work is the intuitive statement
that the geometric median of a set of points
is in their convex hull. 
The geometric median of $p_1,\ldots,p_m \in \R^n$
is the minimizer of the convex function
$$\R^n \ni \xi \mapsto \sum_{i \in [m]} \|\xi-p_i\|_2;$$
it is unique when the points are not collinear (see e.g.~\cite{Minsker_2015} and references within).
The geometric median is also known as the Fermat-Weber point,
the Torricelli point, Haldane’s median, and by several other names. 
It is important in statistics because its breakdown point in $50 \%$;
i.e., we need to change at least half of the mass to send the geometric median to infinity~\cite{lopuhaa1991breakdown}.
It is important in facility location because it captures minimizing
the cost of transportation. 
Consequently, 
it has a very rich history 
and there are many related algorithmic results
(see e.g.~\cite{drezner2002weber} and the long list of references within).

The geometric median makes sense over general norms. 
Let $K \subset \R^n$ be a symmetric convex body 
(i.e., compact with non-empty interior
so that $K=-K$), and 
denote by $\|\cdot\|_K$ the norm defined by $K$:
$$\| \xi \|_K = \inf \{ t > 0 : \xi \in t K\}.$$
We also allow weights for the points.
For a finite set of points $P \subset \R^n$, 
denote by $\Delta_P$ the collection of probability distributions on $P$.
The geometric median $\GM_{K}(W)$ of $W$ over $K$
is the set of minimizers of the convex function
$$F_{K,W}(\xi)=  \sum_{p \in P} W(p) \|\xi-p\|_K.$$
The geometric median is always non-empty and convex. 
If $K$ is strictly convex and the points in $P$ are not collinear then
it is a unique point~\cite{Minsker_2015}.

It is intuitively obvious that
the geometric median is in the convex hull $\conv(P)$
because norms are convex.
Let us say that $K$ is {\em intuitive} if
there always exists a geometric median is in the convex hull of the data points;
for every $W \in \Delta_P$,
$$\GM_K(W) \cap \conv(P) \neq \emptyset.$$
Let us start by describing some intuitive norms. 

\begin{obs*}
All centered $\ell_2$ balls are intuitive.
\end{obs*}

\begin{proof}
If $\xi$ is outside $\conv(P)$
then there is a separating hyperplane with a normal $\ell_2$ unit vector $u$
so that $\ip{u}{\xi-p}  > 0$ for all $p \in P$.
For every $p \in P$ and $t \in \R$,
$$\|\xi - t u - p\|_2 = \sqrt{\|\xi-p\|_2^2 - 2 t \ip{u}{\xi-p} + t^2 }.$$
Now, if $t>0$ is small enough then
the point $\xi-tu$ 
has a smaller $F_{\ell_2,W}$ value than $\xi$. 
\end{proof}

We can obtain more intuitive norms, because
being intuitive is invariant under linear transformations.

\begin{obs*}If $K \subset \R^n$ is intuitive and $T : \R^n \to \R^n$
is an invertible linear map,
then $TK$ is intuitive. 
\end{obs*}

\begin{proof}
For every $\xi$, 
\begin{align*}
F_{TK,W}(\xi)
& = \sum_{p \in P}  W(p) \|\xi - p\|_{TK} \\
& = \sum_{p \in P}  W(p) \|T^{-1}(\xi - p)\|_{K} \\
& = F_{K,W'}(T^{-1} \xi),
\end{align*}
for $P' = T^{-1} P$ and $W' = W \circ T$. 
Because $K$ is intuitive,
there is a minimizer $x$ of $F_{K,W'}$ in $\conv(P')$.
The point $\xi = T x$ is a minimizer of $F_{TK,W}$ 
and $\xi \in \conv(P)$.
\end{proof}

Because linear maps of balls are ellipsoids, 
we can conclude that all Euclidean norms are intuitive
(a.k.a.\ finite dimensional Hilbert spaces). 

\begin{cor*}
All centered ellipsoids are intuitive.
\end{cor*}

The first theorem we prove is that in the plane
all norms are intuitive.

\begin{theorem}
\label{thm:2}
Every symmetric convex body $K \subset \R^2$ is intuitive. 
\end{theorem}

What about higher dimensions?
Our main result is that in higher dimensions, 
the only intuitive norms are Euclidean.

\begin{theorem}
\label{thm:geq3}
Let $K \subset \R^n$ be a symmetric convex body for $n \geq 3$.
Then,
\begin{center}
 $K$ is intuitive $\iff$ $K$ is an ellipsoid.
\end{center}
\end{theorem}

Theorem~\ref{thm:geq3} shows, in particular, that polytopes are never intuitive (for $n \geq 3$).
More generally, that if a norm is intuitive then it must be strictly convex
and smooth. 

It is also natural to consider a stronger definition
where instead of ``there is a geometric median in the convex hull''
we require ``every geometric median is in the convex hull''.
Instead of $\GM_K(W) \cap \conv(P) \neq \emptyset$,
we require $\GM_K(W) \subseteq \conv(P)$.
This stronger definition never holds when the boundary of $K$
has a flat part even just in the case that $P$ comprises two points. 
In particular, this stronger definition never holds for polytopes. 
The reason is that if the boundary of $K$ contains a flat part $F$
then for every $p$ in the interior of $F$,
the geometric median of $P = \{0,p\}$ contains an open set
but $\conv(P)$ is one-dimensional. 
For example, in the plane the geometric median of $(0,1)$
and $(1,0)$ over the $\ell_1$ norm is the entire unit square.
Because ellipsoids are strictly convex, however,
Theorem~\ref{thm:geq3} shows that in dimensions at least three,
the two definitions are equivalent.

We can reduce the case $n >3$ in Theorem~\ref{thm:geq3} to the case $n=3$ as follows. 
A section of $K$ is a set of the form $K \cap H$
for a linear space $H \subset \R^n$.
The dimension of the section is the dimension of $H$.
We think of the section as a subset of the Euclidean space~$H$.

\begin{obs*}
If $K$ is intuitive then every section of $K$ is intuitive. 
\end{obs*}

\begin{proof}
Assume $K$ is intuitive. 
If $H$ is a subspace and $\xi \in H$ then
$\|\xi\|_{K \cap H} = \|\xi\|_K$. 
So, if $P \subset H$ and $W \in \Delta_P$
then 
\begin{equation*}
\emptyset \neq \GM_{K}(W) \cap \conv(P) 
= \GM_{K \cap H}(W) \cap \conv(P).
\qedhere
\end{equation*}
\end{proof}

There are many known characterizations of ellipsoids
(see e.g.~\cite{petty1983ellipsoids,soltan2019characteristic}).
The following well-known lemma provides one such characterization
(a concrete reference is Theorem (3.1) in~\cite{petty1983ellipsoids}).
In Section~\ref{sec:3} we mention a second characterization 
that we shall use.
The lemma can, e.g., be proved by 
Jordan and von Neumann's result that 
$K$ is an ellipsoid iff the parallelogram law
$\|x+y\|_K^2+\|x-y\|_K^2 = \|x\|_K^2 + \|y\|_K^2$ hold for all $x,y \in \R^n$,
which is a two-dimensional condition.

\begin{lemma*}
A convex body $K \subset \R^n$ is an ellipsoid iff
every $3$-dimensional section of $K$ is an ellipsoid.
\end{lemma*}

%\begin{proof}
%Fix $x,y \in \R^n$, and
%let $H$ be a $3$-dimensional linear subspace that contains $x,y$.
%By assumption $K \cap H$ is an ellipsoid
%so the parallelogram law holds for $K \cap H$.
%Because the $K$-norm of $x,y,x+y,x-y$ is equal to their $K\cap H$-norm,
%the parallelogram equality over $K$ holds for $x,y$.
%\end{proof}

The observation and the lemma imply that in order to prove
Theorem~\ref{thm:geq3} it suffices to prove the following. 

\begin{theorem}
\label{thm:3}
If $K \subset \R^3$ 
is intuitive then $K$ is an ellipsoid.
\end{theorem}

\begin{remark}
The proof of the theorem works for a more general class
of norms than ``intuitive norms''. 
It works for all symmetric convex bodies $K$
so that for all sets $P$ of three points and $W \in \Delta_P$,
we have that $\GM_{K}(W) \cap A_P \neq \emptyset$,
where $A_P$ is the affine span of $P$. 
All convex bodies that satisfy this condition are ellipsoids. 
\end{remark}

\section{Sub-gradients}

To argue about general norms, we need the language of sub-gradients. 
In this section, we recall what are sub-gradients and describe some of their basic properties.
The sub-gradient of $f : \R^n \to \R$ at $\xi \in \R^n$ is 
$$\partial f(\xi) = \{ g \in \R^n : \forall y \in \R^n \ f(y) \geq f(\xi) + \ip{g}{y-\xi}\}.$$
If $f$ is convex then $\partial f(\xi)$ is non-empty and convex for all $\xi$. 
If $f$ is also differentiable at $\xi$ then the sub-gradient is unique
(and is equal to the gradient). 
The sum property says that if $f_1,f_2: \R^n \to \R$ are convex,
then for all $\xi \in \R^n$,
$$\partial(f_1+f_2)(\xi) = \partial f_1(\xi) + \partial f_2(\xi),$$
where $A+B = \{ a+b : a \in A, b \in B\}$ is Minkowski sum. 
For $p \in \R^n$, 
we use the notation 
$$F_{K,p}(\xi) = \|\xi - p\|_K.$$
When $K$ is clear from the context,
we omit the subscript $K$ from $\|\cdot\|_K$,
from $F_{K,W}$ and from $F_{K, p}$

\begin{claim}
\label{clm:symG}
Let $K \subset \R^n$ be a symmetric convex body.
For all $p \in \R^n$, 
$$\partial F_0(p) = - \partial F_p(0).$$
\end{claim}

\begin{proof}
\begin{align*}
& g \in \partial F_0(p) \\
\iff & \forall y \in \R^n \ \|y\| \geq \|p\| + \ip{g}{y-p} \\
\iff & \forall y \in \R^n \ \|y\| \geq \|p\| + \ip{-g}{y+p} \\
\iff &  \forall y \in \R^n \ \|y-p\| \geq \|0-p\| + \ip{-g}{y-0} \\
\iff & - g \in \partial F_p(0) . \qedhere
\end{align*}
\end{proof}

\begin{claim}
\label{clm:symG1}
Let $K \subset \R^n$ be a symmetric convex body.
For all $p \in \R^n$, 
$$\partial F_0(p) = - \partial F_0(-p).$$
\end{claim}

\begin{proof}
\begin{align*}
& g \in \partial F_0(p) \\
\iff & \forall y \in \R^n \ \|y\| \geq \|p\| + \ip{g}{y-p} \\
\iff & \forall y \in \R^n \ \|y\| \geq \|p\| + \ip{g}{-y-p} \\
\iff & \forall y \in \R^n \ \|y\| \geq \|- p\| + \ip{-g}{y-(-p)} \\
\iff & - g \in \partial F_0(-p) . \qedhere
\end{align*}
\end{proof}

\begin{claim}
\label{clm:Gc}
Let $K \subset \R^n$ be a symmetric convex body.
For all $p \in \R^n$ and $c>0$, 
$$\partial F_0(c p) = \partial F_0(p).$$
\end{claim}

\begin{proof}
\begin{align*}
& g \in \partial F_0(c p) \\
\iff & \forall y \in \R^n \ \|y\| \geq \|c p\| + \ip{g}{y-c p} \\
\iff & \forall y \in \R^n \ \|cy\| \geq \|c p\| + \ip{g}{c y-c p} \\
\iff & \forall y \in \R^n \ \|y\| \geq \| p\| + \ip{g}{y-p} \\
\iff & g \in \partial F_0(p) . \qedhere
\end{align*}
\end{proof}

The boundary of $K$ is denoted by $\partial K$;
we have $\partial K = \{ p : \|p\|_K=1\}$. 
For $p \in \partial K$,
a tangent to $K$ at $p$ is a supporting hyperplane $H$ to $K$
so that $p \in H$.

\begin{claim}
\label{clm:g}
Let $K \subset \R^n$ be a symmetric convex body, and
let $p \in \partial K$.
Every $g \in \partial F_0(p)$ is a normal to a tangent to $K$ at $p$.
\end{claim}

\begin{proof}
The hyperplane
$\{y : \ip{g}{y-p} = 0\}$ contains $p$,
and the vector $g$ is normal to it. 
Every $y$ of norm less than one
satisfies $1 > \|y\| \geq \|p\| + \ip{g}{y-p}$ with $\|p\|=1$ so
$\ip{g}{y-p} < 0$. 
\end{proof}

\section{The plane}

\begin{proof}[Proof of Theorem~\ref{thm:2}]
Fix a norm $\|\cdot\|  =\|\cdot\|_K$ on the plane. 
Let $P \subset \R^2$ be finite and $W \in \Delta_P$.
Assume towards a contradiction that $\GM_K(W) \cap \conv(P) = \emptyset$.
The sets $\GM_K(W)$ and $\conv(P)$ are convex
so we can use the convex separation theorem. 
We can assume, by translation, 
that there is a line $\ell$ passing through
$0 \in \GM_K(W)$ and separating $\GM_K(W)$ and $\conv(P)$.
By rotation, we can assume that $\ell$ is the $e_1$-axis and that $P$
is strictly above the $e_1$-axis. 
Let $x_*$ is the intersection point of the $e_1$-axis and $\partial K$;
see an illustration in Figure~\ref{fig:1}.

\begin{figure}
\centering

    \begin{tikzpicture}
%         \draw[decoration={closed contour},decorate] plot[smooth cycle] coordinates {(0,0) (2,0) (3,1) (0,2)};
\draw[->,ultra thin] (-3,0)--(3,0) node[right]{$e_1$};
\draw[->,ultra thin] (0,-1.7)--(0,1.7) node[above]{$e_2$};
\draw plot[smooth cycle] coordinates {(-1,1) (1,1) (2,0.3) (1,-1) (-1,-1) (-2,-0.3)};
\draw (-1.9,0) node {\textbullet};
\draw (1.9,0) node {\textbullet};
\draw[->,thick,dotted] (1.3,-1)--(2.7,1.3) node{};
      \end{tikzpicture}%
      \caption{A two-dimensional symmetric convex body $K$. 
      The points $P$ are in the upper half plane. 
      The two marked 
      points are $x_*,-x_*$. The dotted tangent line is in direction $u$.}
      \label{fig:1}
  \end{figure}
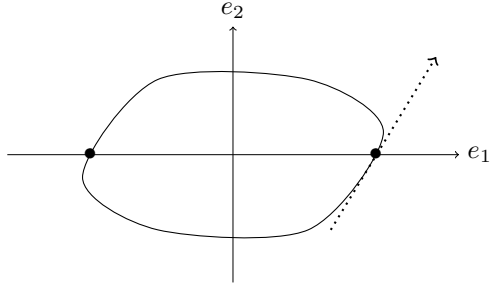

First order optimality implies that $0 \in \partial F_W(0)$.
By the sub-gradient sum property 
and by Claim~\ref{clm:symG},
we can deduce that for each $p \in P$,
there is $g_p \in \partial F_0(p) = - \partial F_p(0)$ so that
\begin{align}
\label{eqn:foo}
\sum_{p \in P} W(p) g_p = 0.
\end{align}
If we delete $x_*,-x_*$ from $\partial K$, then $\partial K$
is broken into two connected components. 
Let $Q \subset \partial K$ be the set of points of the form $\frac{p}{\|p\|}$ for $p \in P$.
All points in $Q$ belong to the same connected component,
which we denote by $\Gamma$ and
is homeomorphic to a line segment.
Let $u$ be a direction of a tangent
to $K$ at $x_*$.
Direct $u$ so that its inner product with $e_2$ is positive
(it cannot be zero because $K$ is symmetric and
has non-empty interior).

We claim that 
$\ip{g_z}{u} \geq 0$ for all $z \in \Gamma$ and $g_z \in \partial F_0(z)$.
Indeed, let $z \in \Gamma$.
The intersection of $K$ and
the line parallel to $u$ through $z$ contains an open set.
So, for some $\eps > 0$ we have $z - \eps u \in K$.
As in the proof of Claim~\ref{clm:g},
the body $K$ is contained in
$\{y : \ip{g_p}{y-z} \leq 0\}$.
Therefore, $\ip{g_p}{z - \eps u-z} \leq 0$
which implies $\ip{g_p}{u} \geq 0.$

By Claim~\ref{clm:Gc},
it follows that for all $p \in P$, we have
$\ip{g_p}{u} \geq 0$.
Equation~\eqref{eqn:foo} now implies that
for all $p \in P$,
$$\ip{g_p}{u} =0.$$
Let $Z$ be the set of $z \in \partial K$ so that there is 
$g_z \in \partial F_0(z)$ so that $\ip{g_z}{u} =0$.
The set $Z$ is symmetric, closed and consists of two parts $Z_1$ and $Z_2 = -Z_1$.
The set $Z_1$ can be a point or an interval that is parallel to $u$.
If $Q \subset Z_1$, then all vectors in $\{g_p : p \in P\}$ point in the same direction
so they cannot sum to zero,
which is a contradiction to~\eqref{eqn:foo}.
A similar argument holds for $Z_2$.
Because $Q \subset Z_1 \cup Z_2$,
it follows that both $Q \cap Z_1$ and $Q \cap Z_2$ are non-empty.
Because $Q$ is contained in the open upper-half plane,
each of $Z_1,Z_2$ has a non-empty intersection with the open upper-half plane.
By symmetry, there is a non-zero intersection also with the open lower-half plane.
It follows that the point $\eps u$ for some small $\eps >0$
satisfies $\|\eps u - p\| = \| 0 - p\|_K$ for all $p \in P$.
Now, $\eps u \in \GM(W)$, which is again a contradiction.
\end{proof}

\section{Three dimensions}

\label{sec:3}

We describe the proof of Theorem~\ref{thm:3} in two stages.
We first consider the smooth case, and only then the general case.
The reason is that the smooth case already contains
the heart of the proof, but the general case involves several
technical complications. 
In both cases, we rely on a characterization of ellipsoids by Montejano and Morales \cite{morales2003variations},
which in turn relies on a theorem of Blaschke.

Let $\mathbb{S}^2 \subset \R^3$ be the two-dimensional sphere.
For $u \in \mathbb{S}^2$, we denote by $u^\perp$
the two-dimensional space in $\R^3$ that is orthogonal to $u$. 
For a convex body $K \subset \R^3$ and $H = u^\perp$,
we think of $H$ as $\R^2$ and of $F_{K \cap H,0}$ as a map $H \to \R$. 
For $\ell \in \mathbb{S}^2$,
denote by $\mathcal{T}(K, \ell)$ the union of all lines
that are tangent to $K$ and parallel to $\ell$.
The shadow boundary in direction $\ell$ is $\mathcal{S}(K, \ell) = \mathcal{T}(K, \ell) \cap K$.
For a subspace $H$, the boundary of $H \cap K$ is denoted by
$\partial(H \cap K) \subset H$.

\begin{prop}[\cite{morales2003variations}]
\label{prop:MM}
Let $K \subset \R^3$ be a symmetric convex body. 
If for every $\ell \in \mathbb{S}^2$,
there is $u = u_\ell \in \mathbb{S}^2$ so that $\partial(H \cap K) \subset \mathcal{S}(K,u)$
with $H = \ell^\perp$
then $K$ is an ellipsoid. 
\end{prop}

The following lemma and proposition 
are the main steps in the proof of the theorem. 

\begin{lemma}
\label{lem:xyW}
Let $L \subset \R^2$ be a symmetric convex body.
Let $x,y \in \partial L$ be so that there are $g_x \in \partial F_0(x)$
and $g_y \in \partial F_0(y)$ that span $\R^2$.
For all $z \in \partial L$,
there is $W \in \Delta_P$ for $P \subset \{\pm x, \pm y, z\}$ 
so that $W(z)>0$ and
$$0 \in \GM_L(W).$$
\end{lemma}

\begin{proof}
Let $g_z \in \partial F_0(z)$
and write $$g_z = a_x g_x + a_y g_y$$ with $a_x,a_y \in \R$.
By Claim~\ref{clm:symG1}, we have $- g_x \in \partial F_0(-x)$ and similarly for $y$.
There is therefore $P \subset \{\pm x, \pm y, z\}$ of size at most three
and $W \in \Delta_P$ so that $W(z)>0$ and
$\sum_{p \in P} W(p) g_p = 0$.
By Claim~\ref{clm:symG}, we have
$- g_p \in \partial F_p(0)$.
By the sub-gradient sum property,
$\partial F_W(0) = \sum_{p \in P} W(p) \partial F_p(0)$.
It follows that
$0 \in \partial F_W(0)$ and so  $0$ is a minimizer of $F_W$.
\end{proof}

\begin{prop}
\label{prop:smooth}
Let $K \subset \R^3$ be intuitive.
Let $\ell \in \mathbb{S}^2$,
let $H = \ell^\perp$ and $L = K \cap H$.
Assume that
there are $x,y \in \partial L$ so that $F_{K,0}$
is differentiable at $x,y$ and there are $g_x \in \partial F_{L,0}(x)$
and $g_y \in \partial F_{L,0}(y)$ that span $H$.
Then,
there is $u = u_\ell \in \mathbb{S}^2$ so that
$\partial(L) \subset \mathcal{S}(K,u)$.
\end{prop}

\begin{proof}
By Lemma~\ref{lem:xyW} with $H$ as $\R^2$,
for all $z \in \partial L$,
there is $W \in \Delta_P$ for $P \subset \{\pm x, \pm y, z\}$ 
so that $W(z)>0$ and
$$0 \in \GM_L(W).$$
Because $K$ is intuitive, there is $\xi \in \GM_K(W) \cap H$.
It follows that 
$$F_{K,W}(0) = F_{L,W}(0) \leq F_{L,W}(\xi) = F_{K,W}(\xi)$$
and in particular
$$0 \in \GM_K(W).$$
By first order optimality in $\R^3$
and the sub-gradient sum property,
$$0 = W(z) G_z + \sum_{p \neq z} W(p) G_p $$
where $G_p \in \partial F_{K,0}(p) = -\partial F_{K,p}(0)$ for $p \in P$.
Because $F_{K,0}$ is differentiable,
the sub-gradients $G_p$ for $p \in \{\pm x,\pm y\}$ are fixed (do not depend on $z$).
The two vectors $G_x,G_y$ span a vector space $V = V_\ell$ in $\R^3$.
Let $u \in \mathbb{S}^2$ be orthogonal to $V$.
Because $G_z \in V = u^\perp$, we know that $u \in (G_z)^\perp$
so $z \in \mathcal{S}(K,u)$.
\end{proof}

\begin{proof}[Proof of Theorem~\ref{thm:3} smooth case]
Let $K \subset \R^3$ be intuitive
and assume that $F_{K,0}$ is differentiable. 
We are going to prove that for all $\ell \in \mathbb{S}^2$,
there is $u = u_\ell \in \mathbb{S}^2$ so that
$\partial(K \cap H) \subset \mathcal{S}(K,u)$ where $H = \ell^\perp$.
Once we prove this, Proposition~\ref{prop:MM} completes the proof.

Fix $\ell$ and $H = \ell^\perp$. 
The set $L = K \cap H$ is a convex body in $H$ and so 
there are $x,y \in \partial L$ so that there are $g_x \in \partial F_{L,0}(x)$
and $g_y \in \partial F_{L,0}(y)$ that span $H$.
Because $K$ is smooth,
we know that $F_{K,0}$ is differentiable at $x,y$.
Proposition~\ref{prop:smooth} completes the proof. 
\end{proof}

In the proof for the smooth case, 
we located points $x,y$ where $F_{K,0}$ is differentiable
(when we used Proposition~\ref{prop:smooth}).
The reason is that we needed the sub-gradients $G_x,G_y$
to not depend on $z$ so that the space $V$
is ``the same for all $z$''. 
When $K$ is smooth, this is automatic, 
but in general we need the following claim. 

\begin{claim}
\label{clm:denseH}
Let $K \subset \R^3$ be a symmetric convex body.
Then, there is a dense set $\cU \subset \mathbb{S}^2$ so that
for all $u \in \cU$, there are $x,y \in \partial(H \cap  K)$ with $H = u^\perp$
so that $F_{K,0}$ is differentiable at $x,y$,
and there are $g_x \in \partial F_{K \cap H,0}(x)$
and $g_y \in \partial F_{K \cap H,0}(y)$ that span $H$.
\end{claim}

\begin{proof}
Denote by $D$ the set of points that $F_{K,0}$ is differentiable at.
By homogeneity of norms, if $x \in D$ then $r x \in D$ for all $r>0$. 
If $B$ is a centered Euclidean ball of radius one,
then the measure of $B \setminus D$ is zero (see e.g.~\cite{rockafellar1997convex}).
Let $\sigma$ be the uniform measure on $\mathbb{S}^2$.
By integration in polar coordinates,
$$0 = \int_{[0,1]} \int_{u \in \mathbb{S}^2} r^2 1_{ru \not \in D}
d \sigma dr = \int_{[0,1]} r^2 \int_{u \in \mathbb{S}^2}  1_{u \not \in D}
d \sigma dr =  \frac{\sigma( \mathbb{S}^2 \setminus D)}{3} ,$$
and so $\sigma(D \cap \mathbb{S}^2) = 1$.
Denote by $\mathbb{S}_u$ the unit circle in $u^\perp$,
and denote by $\sigma_u$ the uniform measure on this circle.
It follows that 
$\sigma_u ( D \cap u^\perp) = 1$ for $\sigma$-almost-surely every 
$u \in \mathbb{S}^2$.

Now, fix $u_0 \in \mathbb{S}^2$,
and let $N$ be an open neighborhood of $u_0$.
By the above, there is $u \in N$
so that $\sigma_u ( D \cap u^\perp) = 1$.
Let $H = u^\perp$, let $x \in D \cap \partial(H \cap  K)$,
and let $g_x \in \partial F_{K \cap H,0}(x)$.
Let $Y$ be the set of $y \in \partial(H \cap  K)$ so that 
there is $g_y \in \partial F_{K \cap H,0}(y)$ that is linearly independent of $g_x$.
Because $K \cap H$ is convex and bounded, 
the set $Y$ contains an open set.
So, there is $y \in D \cap Y$.
\end{proof}

\begin{proof}[Proof of Theorem~\ref{thm:3} general case]
The proof start as in the smooth case. 
Let $K \subset \R^3$ be intuitive. 
Let $\cU$ be the set given by Claim~\ref{clm:denseH}.
By Proposition~\ref{prop:smooth},
for all $\ell \in \cU$,
there is $u = u_\ell \in \mathbb{S}^2$ so that
$\partial(K \cap H) \subset \mathcal{S}(K,u)$ where $H = \ell^\perp$.
It remains to move from the dense set $\cU$ to all of 
the sphere,
because then Proposition~\ref{prop:MM} implies that $K$
is an ellipsoid. 
This move is achieved by continuity.

Let $\ell \in \mathbb{S}^2$,
and let $\ell_1,\ell_2,\ldots \in \cU$ tend to $\ell$. 
By moving to a subsequence if needed,
the vectors $u_{\ell_i}$ tend to a limit which we denote by $u_\ell$.
Fix $z \in \partial(K \cap H)$ with $H = \ell^\perp$.
There is $z_i \in \ell_i^\perp$ so that the $z_i$'s tend to $z$.
Again, moving to a sub-sequence (which could depend on $z$) if needed,
the sub-gradients $G_{z_i}$ tend to a limit $G_z$.
Because for all $i$, 
$$\forall y \in \R^n \ F_{K,0}(y) \geq F_{K,0}(z_i) + \ip{G_{z_i}}{y-z_i},$$
we see that $G_z \in \partial F_{K,0}(z)$.
Because $G_{z_i} \in (u_{\ell_i})^\perp$,
we also have that $G_z \in (u_\ell)^\perp$.
%$$\ip{G_z}{u_H}
%= \ip{G_z - G_{z_i}}{u_H} + \ip{G_{z_i}}{u_H}
%= \ip{G_z - G_{z_i}}{u_H} + \ip{G_{z_i}}{u_H-u_{H_i}}$$
It follows that $z \in \mathcal{S}(K,u_\ell)$.

\end{proof}

\bibliographystyle{amsplain}
\bibliography{intui}   

\providecommand{\bysame}{\leavevmode\hbox to3em{\hrulefill}\thinspace}
\providecommand{\MR}{\relax\ifhmode\unskip\space\fi MR }
\providecommand{\MRhref}[2]{%
  \href{http://www.ams.org/mathscinet-getitem?mr=#1}{#2}
}
\providecommand{\href}[2]{#2}
\begin{thebibliography}{1}

\bibitem{drezner2002weber}
Z.~Drezner, K.~Klamroth, A.~Sch{\"o}bel, and G.~O. Wesolowsky, \emph{The
  {W}eber problem}, Facility location: Applications and theory (2002), 1--36.

\bibitem{lopuhaa1991breakdown}
H.~P. Lopuhaa and P.~J. Rousseeuw, \emph{Breakdown points of affine equivariant
  estimators of multivariate location and covariance matrices}, The Annals of
  Statistics (1991), 229--248.

\bibitem{Minsker_2015}
S.~Minsker, \emph{Geometric median and robust estimation in banach spaces},
  Bernoulli \textbf{21} (2015), no.~4.

\bibitem{morales2003variations}
I.~Montejano and E.~Morales, \emph{Variations of classic characterizations of
  ellipsoids and a short proof of the false centre theorem.}, Mathematika
  \textbf{54} (2007), no.~1, 37--42.

\bibitem{petty1983ellipsoids}
C.~M. Petty, \emph{Ellipsoids}, Convexity and its Applications (1983),
  264--276.

\bibitem{rockafellar1997convex}
R.~T. Rockafellar, \emph{Convex analysis}, vol.~11, Princeton University Press,
  1997.

\bibitem{soltan2019characteristic}
V.~Soltan, \emph{Characteristic properties of ellipsoids and convex quadrics},
  Aequationes mathematicae \textbf{93} (2019), 371--413.

\end{thebibliography}

\end{document}